\documentclass{amsart}
\usepackage[english]{babel}
\usepackage[utf8]{inputenc}
\usepackage{amsmath, amssymb,epic,graphicx,mathrsfs,enumerate}
\usepackage[all]{xy}
\usepackage{color}
\usepackage{comment}
\usepackage{enumitem}
\usepackage{hyperref}
\usepackage[]{frontespizio}

\usepackage{amsmath}
\usepackage{amsthm}
\usepackage{amssymb}
\usepackage{latexsym}
\usepackage{epsfig}

\DeclareMathOperator{\perm}{Sym}
 
\DeclareMathOperator{\frat}{Frat} 
\DeclareMathOperator{\fit}{Fit}

\newcommand{\N}{\mathbb N}

\newtheorem{thm}{Theorem}
\newtheorem{cor}[thm]{Corollary}
 \newtheorem{lemma}[thm]{Lemma}

\newtheorem{question}[thm]{Question}

\numberwithin{equation}{section}

\renewcommand{\footnote}{\endnote}
\newcommand{\ignore}[1]{}\makeglossary

\begin{document}
	\bibliographystyle{amsplain}

\title{Generating hypergraphs of finite groups}

	\author{Andrea Lucchini}
\address{Universit\`a degli Studi di Padova\\  Dipartimento di Matematica \lq\lq Tullio Levi-Civita\rq\rq\\ Via Trieste 63, 35121 Padova, Italy\\email: lucchini@math.unipd.it}
\thanks{Project funded by the EuropeanUnion – NextGenerationEU under the National Recovery and Resilience Plan (NRRP), Mission 4 Component 2 Investment 1.1 - Call PRIN 2022 No. 104 of February 2, 2022 of Italian Ministry of University and Research; Project 2022PSTWLB (subject area: PE - Physical Sciences and Engineering) " Group Theory and Applications"}
\begin{abstract}In a recent paper Cameron, Lakshmanan and Ajith \cite{cam}
began an exploration of hypergraphs defined on algebraic structures, especially groups, to investigate whether this can add a new perspective.	Following their suggestions, we consider suitable hypergraphs encoding the generating properties of a finite group. In particular, answering a question asked in their paper, we classified the finite solvable groups whose generating hypergraph is the basis hypergraph of a matroid.
\end{abstract}
\maketitle

\section{Introduction}
An undirected hypergraph $H$ is a pair $H = (V,E)$ where $V$ is a set of elements called nodes
or vertices and $E$ is a non-empty subset of $\mathcal P(V)$ (power set of $V$) called hyperedges.

\

Following the suggestions and considerations of Cameron, Lakshmanan and Ajith in their recent paper \cite{cam}, we concentrate our attention on hypergraphs that encode  the generating properties of finite groups.

\

The generating  graph of a finite group $G$ is the graph defined on the elements of $G$ in such a way that two distinct vertices are connected by an edge if and only if they generate $G.$  It was defined by Liebeck and
Shalev in \cite{LS}, and has been further investigated by
many authors: see for example \cite{ bglmn,  bucr, spre, tra, CL3,  fu,  LM2, luma, colva} for some
of the range of questions that have been considered. Clearly the generating graph of $G$ is an edgeless graph if $G$ is not 2-generated. It is natural to think that for non-2-generated finite groups, the concept of generating graph can be generalized by making use of appropriate hypergraphs. We propose two possible definitions in this direction. For a finite group $G$ denote by $d(G)$ the smallest size of a generating set of $G$. We define the \emph{generating hypergraph} $\Gamma(G)$ associated to a finite group $G$ as the undirected hypergraph whose vertices are the elements of $G$ and where the hyperedges are the generating sets of $G$ of size $d(G).$ A minimal generating set of $G$ can have a size larger then $d(G)$, so we may also consider another hypergraph $\Delta(G)$, whose vertices are again the elements of $G$ but where the set of hyperedges consists of all the minimal generating sets. We will use for $\Delta(G)$ the name \emph{global generating hypergraph}. The finite groups $G$ with the property that $\Gamma(G)=\Delta(G)$ have been classified by P. Apise and B. Klopsch \cite{ak}: they are solvable, with a quite restrictive structure.

\

The very particular case when $G$ is cyclic is quite different from the general case; for example $\Gamma(G)$ does not contain hyperedges if we do not admit loops. This is the reason why we will in general exclude cyclic groups from our discussion. 

\

In particular we concentrate our attention on the following aspects. In Section \ref{connesso} we will investigate whether  
the hypergraphs obtained from $\Gamma(G)$ and $\Delta(G)$ by removing their isolated vertices are connected. We will prove that the answer is affirmative for the global generating hypergraph and for the generating hypergraph with the possible exclusion, in the second case, of the 2-generated finite groups. In Section \ref{riconosco} we will indicate how relevant structural properties of a finite group $G$ (for example the $p$-solvability property) can be detected from the knowledge of the global generating hypergraph $\Delta(G).$ Finally in Section 	\ref{matroidi}, answering a question asked \
in \cite{cam}, we classify the finite solvable groups whose generating hypergraph is the basis hypergraph of a matroid.

\section{Connectedness of the generating hypergraphs}\label{connesso}

A path in a hypergraph $H = (V,E)$ between two distinct vertices $x_1$ and $x_k$ is a sequence $x_1, e_1, . . . , x_{k-1}, e_{k-1}, x_k$ with the following properties:
\begin{enumerate} 
\item $x_1, . . . , x_k$ are distinct vertices,
\item $e_1, . . . , e_{k-1}$ are hyperedges (not necessarily distinct),
\item $x_j , x_{j+1} \in e_j$ for $j = 1, 2, \dots , k-1.$
\end{enumerate}
Two vertices $v,w \in V$ are said to be connected in $H$ if there exists a $(v,w)$-path in $H$. A hypergraph $H$ is connected
if every pair of vertices $v,w \in V$ is connected in $H$. An obstacle to the connectedness of the generating and global generating hypergraph of a finite group $G$ is the possible presence of isolated vertices (certainty the identify element of $G$ is an isolated vertex in both hypergraphs). It is so interesting to describe the set of the isolated vertices of $\Gamma(G)$ and $\Delta(G)$ and to investigate whether the hypergraphs $\tilde \Gamma(G)$ and $\tilde \Delta(G)$ obtained  from $\Gamma(G)$ and $\Delta(G)$ by removing the isolated vertices are connected. 

In the case of the global generating ipergraph, the answer to the question follows immediately from the results obtained in \cite{ind}.

\begin{thm}\label{indip}
	Let $G$ be a finite group. The set of isolated vertices of $\Delta(G)$ coincides with the Frattini subgroup $\frat(G)$ of $G$ and the hypergraph $\tilde \Delta(G)$ is connected.
\end{thm}
\begin{proof}
In \cite{ind} it is defined and investigated the independence graph of $G$: its vertices are the elements of $G$ and two  vertices $x_1$ and $x_2$	are joined by an edge if and only if $x_1\neq x_2$ and 
	there exists a minimal generating set of $G$ containing $x_1$ and $x_2$. Clearly the independence graph and the global generating hypergraph of $G$ have the same isolated vertices and $\tilde \Delta(G)$ is connected if and only if it is connected the subgraph of the independence graph of $G$ induced by its non-isolated vertices. So the conclusion follows from Lemma 4 and Theorem 1 in \cite{ind}.
\end{proof}

Moving from $\Delta(G)$ to $\Gamma(G)$ the role of the independence graph is played by the rank graph of $G$, in which two distinct vertices $x_1$ and $x_2$ 	are joined by an edge if and only if $x_1\neq x_2$ and 
there exists a generating set of $G$ of size $d(G)$ containing $x_1$ and $x_2$. Again the isolated vertices of the rank graph and of the generating hypergraph are the same, but it is difficult to give a good description of the set of these isolated vertices. It is in general larger of the Frattini subgroup (for example it contains all the elements of any normal subgroup $N$ of $G$ with the property that $d(G)=d(G/N)$) and it is not in general a subgroup. When $d(G)=2,$ \lq generating graph', \lq rank graph' and \lq generating hypergraph' are different names for the same object. It seems reasonable to conjecture that the subgraph of the generating graph of a 2-generated finite group $G$ induced by its non-isolated vertices  is connected. This has been proved under additional assumptions, for example if $G$ is solvable \cite{CL3} or if $G$ is a direct product of simple groups \cite{cl2}, but the question whether this is true in general is quite difficult and still widely open. However the situation is better when $d(G)\geq 3.$ In a recent joint paper with D. Nemmi \cite{rank} we proved that, if $d(G)\geq 3,$ then the subgraph of the rank graph of $G$ induced by its non-isolated vertices is connected. This implies the following theorem:

\begin{thm}Let $G$ be a finite group. If $d(G)\geq 3,$ then the hypergraph $\tilde\Gamma(G)$ is connected.
	\end{thm}

\section{Recognizing a finite group by its generating hypergraphs}\label{riconosco}

For a given $t\in \mathbb N,$ denote by $P_G(t)$ the probability that an element $(g_1,\dots,g_t)$ of $G^t$ is a generating $t$-uple for $G$, i.e $G=\langle g_1,\dots,g_t\rangle.$ If $t\in \mathbb N,$ then $(g_1,\dots,g_t)$ is a generating
$t$-uple of $G$ if and only if $\{g_1,\dots,g_t\}$ contains an hyperedge of $\Delta(G).$ So from the knowledge of the global generating hypergraph $\Delta(G)$ we can determine $P_G(t)$ for every $t\in \mathbb N.$
When we know the value of $P_G(t)$ for every $t \in \mathbb N$ (or at least for $|G|$ different values of $t$),  we can deduce a great amount of information on the structure of the finite group $G$. Let us explain why.  For any finite group $G$ we may define a sequence of integers
$\{a_n(G)\}_{n \in \N}$ as follows:
$$
a_n(G)=\sum_{|G:H|=n}\mu _G(H).
$$
Here $\mu _G$ is the M\"obius function defined on the subgroup
lattice of $G$ as $\mu _G(G)=1$ and $\mu _G(H)=-\sum
_{H<K}\mu _G(K)$ for any $H<G$. 

\begin{lemma}\label{ang}
	Let $G$ be a finite group. Then the knowledge of the hypergraph $\Delta(G)$ allows to determine $a_n(G)$ for every $n\in \mathbb N.$
\end{lemma}
\begin{proof}Let $\sum _{n \in
	\N}{a_n(G)}/{n^s}$ be the Dirichlet generating function
associated with the sequence $\{a_n(G)\}_{n \in \N}$. As it was noticed by P. Hall \cite{hall}, for any $t \in
\N$, the value assumed on $t$ by this Dirichlet polynomial coincides with $P_G(t).$ Let $m=|G|.$ Notice
$(a_1(G),\dots,a_m(G))$ is a solution of the system
$$\begin{cases}
\frac{a_1(G)}{1}+\frac{a_2(G)}{2}+\dots+\frac{a_m(G)}{m}=P_G(1)\\
\frac{a_1(G)}{1}+\frac{a_2(G)}{2^2}+\dots+\frac{a_m(G)}{m^2}=P_G(2)\\
\dots \quad \dots \quad \dots \quad \dots \quad \dots\quad \dots \quad \dots\\
\frac{a_1(G)}{1}+\frac{a_2(G)}{2^m}+\dots+\frac{a_m(G)}{m^m}=P_G(m).\\
\end{cases}$$
Since
$$\det\begin{pmatrix}
1&2&\cdots&m\\
1&2^2&\cdots&m^2\\
\vdots&\vdots&\cdots&\vdots\\
1&2^m&\cdots&m^m
\end{pmatrix}\neq 0,$$
the system has a unique solution, hence the Dirichlet polynomial $P_G(s)$ is uniquely determined from the knowledge of $P_G(1), P_G(2),\dots, P_G(m)$ and, consequently, from the knowledge of the hypergraph $\Delta(G).$
\end{proof}

\begin{thm}\label{psolv}Let $G$ be a finite group and let $p$ be a prime divisor of $|G|.$ From the knowledge of the hypergraph $\Delta(G)$ it is possible to detect whether $G$ is $p$-solvable. In particular  it  can be recognized whether $G$ is solvable.
\end{thm}

\begin{proof}The proof follows from Lemma \ref{ang} combined with \cite[Theorem 1]{dlsol} and \cite[Theorem 1.2]{ea}.  Indeed $G$ is $p$-solvable if and only if $a_{p^cd}(G)=a_{p^c}(G)a_d(G),$ whenever $p$ and $d$ are relatively prime. In particular $G$ is solvable if and only if $a_{rs}(G)=a_r(G)a_s(G),$ whenever $r$ and $s$ are relatively prime. 
\end{proof}

\begin{thm}
	Let $G$ be a finite non-abelian simple group and let $X$ be a finite group. If $\Delta(X)\cong \Delta(G),$ then $X\cong G.$
\end{thm}
\begin{proof}
By Lemma \ref{ang}, $a_n(X)=a_n(G)$ for every $n\in \mathbb N.$ It follows from \cite{easimple} that $X/\frat(X)$ is a finite non-abelian simple group. Moreover $\frat(G)=1,$ hence, by Theorem \ref{indip}, the identity element is the unique isolated vertex of $\Delta(G)$. Thus $\Delta(X)$ has the same property, and consequently $\frat(X)=1.$ Hence $X$ and $G$ are simple groups with the same order. By \cite[Theorem 5.1]{artin} $G$ and $X$ either are $A_2(4)$ and $A_3(2)$ or are $B_n(q)$ and $C_n(q)$ for some
$n\geq 3$ and some odd $q.$ However, a direct computation shows that
$P_2(A_2(4))\neq P_2(A_3(2))$, and in  \cite[Section 7]{colva} it is proved that $P_2(B_n(q))\neq P_2(C_n(q)).$ Hence  $X\cong G.$
\end{proof}

The procedure described in the proof of Theorem \ref{psolv} in order to recognize from the knowledge of the global generating hypergraph  $\Delta(G)$ whether the finite group $G$ is solvable is quite complicate. One has to determine the rational number $P_G(t)$ for several values of $t$ (this is in principle an elementary task but requires a boring enumeration task) and to solve a system with many equations and indeterminates. So it is natural to ask whether there is some other more directed criterion, which relies for example on the topological properties of the hypergraph $\Delta(G).$

\section{Basis hypergraph of a matroid}\label{matroidi}

An important class of hypergraphs consists of the basis hypergraphs of matroids. A basis of a matroid is a maximal independent set. The collection of matroid bases is characterised by the two properties:
\begin{enumerate}
	\item there is at least one basis;
\item (the Exchange Axiom): if $A,B$ are bases and $a \in A \setminus B,$ then there exists $b \in B\setminus A$
such that $(A \setminus \{a\}) \cup \{b\}$ is a basis.
\end{enumerate}

It follows from the Exchange Axiom that any two bases have the same number of elements;
that is, the basis hypergraph is uniform, i.e. all the hyperedges have the same size. So if the global generating graph of a finite group $G$ is the basis hypergraph of a matroid, then $\Delta(G)=\Gamma(G).$ The authors of \cite{cam} asked the following question:
\begin{question}
	Is it possible to describe groups whose generating hypergraph is the basis hypergraph of a matroid?
\end{question}

Motivated by the previous questions, we want to study the finite groups $G$ which satisfy the following property: if $G=\langle x_1,\dots,x_{d(G)}\rangle=\langle y_1,\dots,y_{d(G)}\rangle,$ then for every $1\leq i\leq d(G),$ there exists $j$ such that $$\langle x_1,\dots,x_{i-1},y_j,x_{i+1},\dots,x_{d(G)}\rangle=G.$$ We call this property the \lq minimal generating set exchange property\rq \ (MGSE property). If the previous property fails for the two generating sets
$X=\{x_1,\dots,x_d\}$ and $X=\{y_1,\dots,y_d\}$ we will say that $X$ and $Y$ fail the MGSE property.

\

We start our investigation with the following elementary observation.

\begin{lemma}\label{nofrat}
Let $G$ be a finite group. Then $G$ satisfies the MGSE property if and only if $G/\frat(G)$ satisfies the MGSE property. 
\end{lemma}

We will use the following result by Gasch\"utz \cite{g2} to prove that the MGSE property is inherited by the quotients.

\begin{lemma}\label{gas}
	Let $N$ be a normal subgroup of $G$ and
	$\langle g_1N,\dots,g_rN\rangle = G/N.$ If $r \geq d(G),$  then
	we can find $u_1 \dots u_r$ in $N$ so that $\langle g_1u_1,\dots,g_ru_r\rangle=G.$
\end{lemma}

\begin{lemma}\label{quoz}
	Let $G$ be a finite group. If $G$ satisfies the MGSE propery, then $G/N$ satisfies MGSE property for every normal subgroup $N$ of $G.$
\end{lemma}
\begin{proof}
	Let $N\unlhd G,$ $u=d(G/N), v=d(G).$ Suppose that the two minimal generating sets $\tilde X=\{x_1N,\dots,x_uN\}$ and $\tilde Y=\{y_1N,\dots,y_uN\}$ fail the MGSE property. We may assume that
	$\{y_jN,x_2N,\dots,x_uN\}\neq G/N$ for every $1\leq j\leq u.$ 
	By Lemma \ref{gas}, there exist $m_1,\dots,m_v,n_1,\dots,n_v \in N,$ such that
	$$G=\langle x_1m_1,\dots,x_um_u,m_{u+1},\dots,m_v\rangle=
	\langle y_1n_1,\dots,y_un_u,n_{u+1},\dots,n_v\rangle.
	$$
	We claim that the minimal generating sets $$X=\{ x_1m_1,\dots,x_um_u,m_{u+1},\dots,m_v\}\text{ and } Y=\{y_1n_1,\dots,y_un_u,n_{u+1},\dots,n_v\}$$ of $G$ fail the MGSE property. Indeed, for every $y\in Y,$ 
	$$\langle yN, x_2m_2N,\dots,x_um_uN,m_{u+1}N,\dots,m_vN\rangle=
	\langle yN, x_2N,\dots,x_uN\rangle\neq G/N.\qedhere
	$$
\end{proof}

By \cite[Theorem 5.1]{cam}, a finite $p$-group satisfies the MGSE property. We are going to prove that this is the unique case in which a non-cyclic finite nilpotent group satisfies the MGSE property.

\begin{lemma}\label{ciclo}Let $G$ be a non-cyclic finite group which satisfies the MGSE property. If $N$ is a normal subgroup of $G$ and $G/N$ is cyclic, then $G/N$ is a $p$-group. In particular $G/G^\prime$ is a $p$-group.
\end{lemma}
\begin{proof}
	Suppose, by contradiction, that $G/N=\langle gN\rangle$ and that $|gN|$ is not a prime-power. Then there exist $g_1,g_2\in G$ such that
	$g=g_1g_2$ and neither $g_1N$ nor $g_2N$ generates $G/N.$ Let $d=d(G).$ By Lemma \ref{gas}, there exist $m_1,\dots,m_d,n_1,\dots,n_d\in N$ such that
	$$G=\langle gm_1,m_2,\dots,m_d\rangle=\langle g_1n_1,g_2n_2,n_3,\dots,n_d\rangle.$$ Then  $A=\{gm_1,m_2,\dots,m_d\}$ and $B=\{g_1n_1,g_2n_2,n_3,\dots,n_d\}$ fail the MGSE property. Indeed there is no $b\in B$ such that $G=\langle b,m_2,\dots,m_d\rangle.$ By Lemma \ref{quoz} this contradicts the fact that $G$ satisfies the MGSE property.
\end{proof}

\begin{cor}\label{nilpo}
Let $G$ be a finite non-cyclic finite nilpotent group. Then $G$ satisfies the MGSE property if and only if $G$ is a $p$-group.
\end{cor}
\begin{proof}It is easy to see that a finite $p$-group satisfies the MGSP propery (see \cite[Theorem 5.1]{cam}). Conversely let $G$ be a finite non-cyclic nilpotent group which satisfies the MGSE property. By Lemmas \ref{quoz} and \ref{ciclo} $G/G^\prime$ is a $p$-group. This implies that $G$ itself is a $p$-group.
\end{proof}

The next results aim to obtain a complete classification of the finite solvable groups satisfying the MGSE property. 

\begin{lemma}\label{cog}
	Let $G$ be a 2-generated solvable non-nilpotent group. If $G$ satisfies the MGSE property, then $G/\frat(G)\cong V\rtimes \langle x \rangle,$ where $x$ has prime order and $V$ is a faithful irreducible
	$\langle x\rangle$-module.
\end{lemma}
\begin{proof}Let $G$ be a 2-generated non-nilpotent solvable group. By \cite[Theorem 1.7]{forb} either $G/\frat(G)\cong V\rtimes \langle x \rangle,$ where $x$ has prime order and $V$ is a faithful irreducible
	$\langle x\rangle$-module, or the generating graph $\Gamma(G)$ of $G$ is not a cograph. In the second case, $\Gamma(G)$ contains four vertices $x_1, x_2, x_3, x_4$ such that the subgroup of $\Gamma(G)$ induced by them is the the four-vertex path $x_1-x_2-x_3-x_4.$ This implies that $A=\{x_1,x_2\}$ and
	$B=\{x_3,x_4\}$ fail the MGSE property, since there is no $b\in A$ such that $\langle x_1,b\rangle=G.$
\end{proof}

\begin{lemma}\label{prima}Assume that $H$ is a non-cyclic finite solvable group and let $V$ be a faithful irreducible $H$-module. Then
	$G=V\rtimes H$ does not satisfy the MGSE property.
\end{lemma}
\begin{proof}
Let $d=d(H).$ By the main theorem in \cite{lm}, $d(G)=d(H).$ Assume $H=\langle h_1,\dots,h_d\rangle.$ By \cite[Lemma 2]{lmsymb}, it is not restrictive to assume that there exists $n\in N$ such that
$G=\langle h_1n_1,h_2,\dots,h_d\rangle.$  Moreover, by \cite[Proposition 2.2]{ln} and its proof, there exist $n_2,\dots,n_d\in N$ such that $G=\langle h_1,h_2n_2,\dots,h_dn_d\rangle.$ We claim that the generating sets $X=\{h_1n_1,h_2,\dots,h_d\}$ and $Y=\{h_1,h_2n_2,\dots,h_dn_d\}$ fail the MGSEP property. Indeed assume that there exists $y\in Y$ such that $\langle y,h_2,\dots,h_d\rangle=G.$ Since $\langle h_1,\dots,h_d\rangle=H<G,$ it must be $y=h_in_i$ for some $2\leq i\leq d.$ But this would implies $\langle h_in_i,h_2,\dots,h_d\rangle
\leq \langle h_2,\dots,h_d\rangle N< HN=G.$
\end{proof}

\begin{lemma}\label{corona}
Let $G$ be a finite non-nilpotent solvable group. If $G$ satisfies the MGSE property, then
$G/\frat(G)\cong N^\delta \rtimes H,$ where $H$ is cyclic of prime order, $N$ is an irreducible $H$-module and $C_H(V)=1.$
\end{lemma}
\begin{proof}It follows from Lemma \ref{nofrat} and the fact that $G$ is nilpotent if and only if $G/\frat(G)$ is nilpotent, that we may assume $\frat(G)=1.$ 
	
	But then  we may write $\fit(G) = N_1 \times \dots \times N_t$ as a product of abelian minimal normal subgroups of $G.$ Since $\frat(G)=1,$ $\fit(G)$ has a complement $H$ in $G.$ Let $C_i=C_H(N_i)$ and $M_i=\prod_{i\neq j}N_j.$ Then $M_iC_i=C_G(N_i)$ is a normal sugroup of $G$ and $G/M_iC_i \cong N_i\rtimes H/C_i$. By Lemma \ref{prima}, $H/C_i$ is cyclic. Since $C_G(\fit(G))=\fit(G),$
	$$H\cong \frac{G}{\fit(G)}\cong \frac{G}{\bigcap_{1\leq i\leq t}C_G(N_i)}\leq \prod_{1\leq i\leq t}\frac{G}{C_G(N_i)}\cong \prod_{1\leq i\leq t}\frac{H}{C_i}.$$ Hence $H$ is abelian and it follows from Lemma \ref{ciclo} that  $H$ is a $p$-group.
There exists $1\leq i\leq t$ such that
	$[H,N_i]\neq 1$ (otherwise $G$ would be abelian and consequently $G=\fit(G))$. This implies in particular $(|N_i|,p)=(|N_i|,|H|)=1.$ As we noticed before, $H/C_i$ is cyclic, so we may assume $H=\langle h_1,h_2,\dots,h_u\rangle$ with $h_j\in C_i$ if $j\geq 2.$ Let $X=N_iH$ and $1\neq n\in N.$ Notice that, since $(|C_i|,|N_i|)=1,$ $\langle hm\rangle=\langle h, m\rangle$ for every $h\in H$ and $m\in N.$ In particular, if $u:=d(H)\geq 3,$ then  $A=\{h_1,\dots,x_{u-2},h_{u-1}n,h_u\}$ and $B=\{h_1,\dots,x_{h-2},h_{u-1},h_un\}$ are two minimal generating sets of $X$ which fail the MGSE property, since there is no $b\in B$ such that $G=\langle h_1,\dots,h_{u-2},b,h_u\rangle$. Since $X$ in an epimorphic image $G$, it follows from Lemma \ref{quoz} that the assumption $u\geq 3$ leads to a contradiction. Hence $d(H)\leq 2.$
	
	Let $J$ be a subset of $I=\{1,\dots,t\}$ with the property that for any $i\in I$ there exists one and only one $j\in J$ such that $N_i$ is $H$-isomorphic to $N_{j}.$ Let $Y=(\prod_{j\in J}N_j)H.$ By Lemma \ref{quoz}, $Y$ satisfies the MGSE property. The maximal subgroups of $Y$ are:
	\begin{enumerate}
		\item $(\prod_{j\in J})M_j\frat(H);$
		\item $(\prod_{j\neq k}M_j)H^y$ with $k\in J$ and $y\in M_j.$
	\end{enumerate}
This implies  that $$\frat(Y)=\bigcap_{j\in J}C_{\frat(H)}(N_j)=\bigcap_{i\in I}C_{\frat(H)}(N_i)=\frat(G)=1.$$
Moreover it follows from \cite[Satz 2 and Satz 3]{ga} that $d(Y)=2.$ But then we may deduce  from Lemma \ref{cog} that $|J|=1$ and $|H|=p.$ 
\end{proof}
	
\begin{lemma}\label{viceversa}
	Let $V$ be a faithful irreducible $\langle x\rangle$-module with $|x|=p$ a prime. Then the semidirect product $G=V^\delta\rtimes \langle x\rangle$ satisfies the MGSE  property.
\end{lemma}
\begin{proof}	
We may identify $V$ with the additive group of a finite field $F$ and $x$ with an element of order $p$ of the multiplicative group of $F.$ Notice that $d(G)=\delta+1.$ Moreover it follows from \cite[Proposition 2.1 and 2.2]{LM2} that
$$\langle(v_{1,1},\dots,v_{1,\delta+1})x_1,\dots,(v_{\delta+1,1},\dots,v_{\delta+1,\delta+1})x_{\delta+1}\rangle=G$$ if and only if
$$\det \begin{pmatrix}
	x_1-1&\cdots&x_{\delta+1}-1\\
	v_{1,1}&\cdots&v_{\delta+1,1}\\
	\vdots&\vdots&\vdots\\
		v_{1,\delta+1}&\cdots&v_{\delta+1,\delta+1}
\end{pmatrix}\neq 0.$$
In particular, to prove that two generating sets $$\begin{aligned}X_1=\{(v_{1,1,1},\dots,v_{1,\delta+1,1})x_{1,1},\dots,
(v_{1,\delta+1,1},\dots,v_{1,\delta+1,1})x_{1,\delta+1}\}\\
X_2=\{(v_{2,1,1},\dots,v_{2,\delta+1,1})x_{2,1},\dots,
(v_{2,\delta+1,1},\dots,v_{2,\delta+1,1})x_{2,\delta+1}\}
\end{aligned}$$
satisfy the MGSE property it suffices to notice that, given the two-matrices
$$A_1=\begin{pmatrix}x_{1,1}-1&\cdots&x_{1,\delta+1}-1\\	v_{1,1,1}&\cdots&v_{1,\delta+1,1}\\
\vdots&\vdots&\vdots\\
v_{1,1,\delta+1}&\cdots&v_{1,\delta+1,\delta+1}
\end{pmatrix},\quad A_2=\begin{pmatrix}x_{2,1}-1&\cdots&x_{2,\delta+1}-1\\	v_{2,1,1}&\cdots&v_{2,\delta+1,1}\\
\vdots&\vdots&\vdots\\
v_{2,1,\delta+1}&\cdots&v_{2,\delta+1,\delta+1}
\end{pmatrix},
$$
for any $i\in \{1,\dots,\delta+1\}$ there exists $j\in \{1,\dots,\delta+1\}$ such that the matrix obtained replacing the $i$-th column of $A_1$ with the $j$-th column of $A_2$ is still invertible.
\end{proof}

We may summarize the results obtained in this section in the following statement:
\begin{thm}
	A non-cyclic finite solvable group $G$ satisfies the MGSE property if and only if $G/\frat(G)\cong N^\delta \rtimes H,$ where $H$ is cyclic of prime order, $N$ is an irreducible $H$-module and $C_H(V)=1.$
\end{thm}

Combining  the previous theorem with the main result of \cite{ak}, it turns out that if a finite group $G$ satisfies the MGSE property, then all the minimal generating sets of $G$ have the same size, or equivalently $\Delta(G)=\Gamma(G).$ However the converse is not true. For example, let $G$ be the solvable subgroup of $\perm(5)$ generated by $x=(2,3,4,5)$ and $y=(1,2,3,5,4).$ Then all the minimal generating sets of $G$ has size $2,$ however $A=\{x^2,xy\}$ and $B=\{x,y\}$ are two generating sets that fails the MGSE property, since $\langle x^2,x\rangle$ and $\langle x^2,y\rangle$ are proper subgroups of $G$.

\

We conjecture that a finite group which satisfies the MGSE property must be solvable. The following is a partial result supporting this conjecture.

\begin{thm}
	Let $G$ be a non-nilpotent finite group. If $G$ satisfies the MGSE property, then $G$ contains a unique maximal normal subgroup, say $N$, and $G/N$ is cyclic of order $p$. In particular $G$ is non-perfect and $G/G^\prime$ is a cyclic $p$-group.
\end{thm}
\begin{proof}
	By Lemmas \ref{quoz} and \ref{ciclo}, it suffices to prove that a finite non-abelian simple group does not satisfies the MGSE property. This follows from the fact that the generating graph of a finite non-abelian simple group contains a 4-path \cite[Corollary 1.10]{forb}.
\end{proof}


\begin{thebibliography}{99}
\bibitem{ak} P. Apisa and B. Klopsch, \emph{A generalization of the Burnside basis theorem}, J. Algebra \textbf{400} (2014), 8--16. 
		\bibitem{bglmn}
		T. Breuer, R.\,M. Guralnick, A. Lucchini, A. Mar\'oti and G.\,P.~Nagy, 
	\emph{Hamiltonian cycles in the generating graphs of finite groups}, 
	{Bull. London Math. Soc.} \textbf{42} (2010), 621--633.
	
	\bibitem{bucr} T. Burness and E. Crestani, \emph{On the generating graph of direct powers of a simple group}, J. Algebraic Combin. \textbf{38} (2013), no. 2, 329--350.
	
\bibitem{spre} T. Burness, R. Guralnick and S. Harper,
\emph{The spread of a finite group},
Ann. of Math. (2) \textbf{193} (2021), no. 2, 619–687.
	
	
\bibitem{cam} P. Cameron, A. Lakshmanan S. and M. V. Ajith, \emph{Hypergraphs defined on algebraic structures}, arXiv:2303.00546.
\bibitem{easimple} E. Damian and A. Lucchini, \emph{The probabilistic zeta function of finite simple groups}, J. Algebra \textbf{313} (2007), no. 2, 957–971.

\bibitem{tra} P. Cameron, A. Lucchini, C. Roney-Dougal, 
\emph{Generating sets of finite groups},
Trans. Amer. Math. Soc. \textbf{370} (2018), no. 9, 6751–6770.

\bibitem{CL3} E. Crestani and A. Lucchini, \emph{The generating graph of a finite soluble groups},
Israel J. Math. \textbf{207} (2015), no. 2, 739--761.
	
\bibitem{cl2} E. Crestani and A. Lucchini, \emph{The non-isolated vertices in the generating graph of a direct powers of simple groups}, J. Algebraic Combin. 37 (2013), no. 2, 249–263.
	
\bibitem{ea} E. Damian and A. Lucchini, 
\emph{Finite groups with p-multiplicative probabilistic zeta function},
Comm. Algebra \textbf{35} (2007), no. 11, 3451–3472.

\bibitem{dlsol} E. Detomi and A. Lucchini,
\emph{Recognizing soluble groups from their probabilistic zeta functions,}
Bull. London Math. Soc. \textbf{35} (2003), no. 5, 659–664.

\bibitem{fu} F. Erdem, 
\emph{On the generating graphs of symmetric groups},
J. Group Theory \textbf{21} (2018), no. 4, 629--649.

\bibitem{g2} W. Gasch\"utz, \emph{Zu einem von B. H. und H. Neumann gestellten Problem}, Math. Nachr. \textbf{14}
(1955), no.~4--6, 249--252.

\bibitem{ga} W. Gasch\"utz, \emph{Die Eulersche Funktion endlicher auflösbarer Gruppen}, Illinois J. Math. \textbf{3} (1959), 469–476 .

\bibitem{hall}
P. Hall, \emph{The eulerian functions of a group}, Quart. J. Math. (1936),
no.~7, 134--151.	

\bibitem{artin} W. Kimmerle, R. Lyons, R. Sandling and D. N. Teague, \emph{Composition factors from the group ring
and Artin’s theorem on orders of simple groups}, Proc. Lond. Math. Soc. (3) \textbf{60} (1990), 89–122.

	\bibitem{LS} M. Liebeck and A. Shalev, \emph{Simple groups, probabilistic methods,
	and a conjecture of {K}antor and {L}ubotzky}, J. Algebra \textbf{184} (1996),
no.~1, 31--57.

\bibitem{ind} A. Lucchini,\emph{The independence graph of a finite group}, Monatsh. Math. \textbf{193} (2020), no. 4, 845–856. 

	\bibitem{LM2} A. Lucchini and A. Mar\'oti, \emph{On the clique number of the generating graph of a finite group}, {Proc. Amer. Math. Soc.} \textbf{137}, No. 10, (2009), 3207--3217.

\bibitem{luma}
A. Lucchini and A. Mar\'oti,
\emph{Some results and questions related to the generating graph of a
finite group},
in {Ischia group theory 2008}, 183--208, 
World Sci. Publ., Hackensack, NJ, 2009.	

\bibitem{colva} A. Lucchini, A. Maróti and C. Roney-Dougal, \emph{On the generating graph of a simple group}, J. Aust. Math. Soc. \textbf{103} (2017), no. 1, 91--103.


\bibitem{lm} A. Lucchini and F. Menegazzo, \emph{Generators for finite groups with a unique minimal normal subgroup}, Rend. Sem. Mat. Univ. Padova \textbf{98} (1997), 173–191. 

\bibitem{lmsymb} A. Lucchini and F. Menegazzo,
\emph{Computing a set of generators of minimal cardinality in a solvable group},
J. Symbolic Comput. \textbf{17} (1994), no. 5, 409–420.


\bibitem{ln}  A. Lucchini and D. Nemmi, \emph{The non-$\mathfrak F$ graph of a finite group}, Math. Nachr. \textbf{294} (2021), no. 10, 1912--1921.

\bibitem{forb} A. Lucchini and D. Nemmi, \emph{Forbidden subgraphs in generating graphs of finite groups},
Algebr. Comb. \textbf{5} (2022), no. 5, 925--946.

\bibitem{rank} A. Lucchini and D. Nemmi, \emph{Connectedness of the rank graph of a finite group}, in preparation.



\end{thebibliography}
\end{document}